\documentclass[11pt]{amsart}

\usepackage{style}
\usepackage[left=3.4cm,right=3.4cm,bottom=3.4cm,top=3.4cm]{geometry}
\usepackage{times}

\newcommand{\longhookrightarrow}{\lhook\joinrel\longrightarrow}

\tikzstyle{nodal}=[circle,draw,fill=black,inner sep=0pt, minimum width=4pt]
\tikzstyle{half-fiber}=[rectangle,draw=black,thick,inner sep=0pt, minimum width=5pt, minimum height=5pt]
\tikzset{double distance = 2pt}

\begin{document}

\author{Gebhard Martin}
\address{Mathematisches Institut \\ Universität Bonn \\ Endenicher Allee 60 \\ 53115 Bonn \\ Germany}
\email{gmartin@math.uni-bonn.de} 

\author{Giacomo Mezzedimi}
\address{Mathematisches Institut \\ Universität Bonn \\ Endenicher Allee 60 \\ 53115 Bonn \\ Germany}
\email{mezzedim@math.uni-bonn.de}

\author{Davide Cesare Veniani}
\address{IDSR \\ Universität Stuttgart \\ Pfaffenwaldring 57 \\ 70569 Stuttgart \\ Germany}
\email{davide.veniani@mathematik.uni-stuttgart.de}

\title{Nodal Enriques surfaces are Reye congruences}

\date{\today}


\begin{abstract}
We show that every classical Enriques surface containing a smooth rational curve is a Reye congruence.
\end{abstract}

\maketitle

\section{Introduction}

Let \(X\) be a classical Enriques surface, that is, a smooth and proper surface over an algebraically closed field of arbitrary characteristic \(p \geq 0\) with \(b_2(X) = 10\) and such that its canonical class~\(K_X\) is numerically, but not linearly, trivial.   

A smooth quadric hypersurface in \(\IP^5\) is isomorphic to the Grassmann variety \(\IG(1,3)\), and an irreducible surface in \(\IG(1,3)\) is classically called a \emph{congruence} of lines in~\(\IP^3\). We say that \(X\) is a \emph{Reye congruence} (\Cref{def:Reye.polarization}) if \(X\) admits a morphism \(\varphi_H\colon X \to \IP^5\), given by a \emph{Fano polarization} \(H\) (\Cref{def:Fano.polarization}), onto a possibly singular, non-degenerate surface \(Y\) of degree \(10\) such that \(\varphi_H\) factors through a smooth quadric hypersurface. 

The surface \(Y\) admits a hyperplane section that splits into two plane cubics and a quartic curve with a smooth rational component (see \cite[(2.4.1)]{Cossec:Reye} or \Cref{lem:existence.Reye.bundle} below). In particular, every Reye congruence is a \emph{nodal} Enriques surface, that is, an Enriques surface that contains a \((-2)\)-curve.

A dimension count (see, e.g., \cite[p.~737]{Cossec:Reye}) shows that Reye congruences form a \(9\)-dimensional family of Enriques surfaces, which therefore coincides generically with the irreducible \(9\)-di\-men\-sion\-al family of nodal Enriques surfaces in the \(10\)-dimensional moduli space of all Enriques surfaces. Despite modern studies of Reye congruences such as \cite{conte.verra} and \cite{Cossec:Reye}, the question whether an arbitrary nodal Enriques surface is a Reye congruence has so far remained open (see \cite[Conjecture~(4.1)]{conte.verra}, \cite[Problem~8]{Dolgachev.Whatisleft}, \cite[p.~126]{DolgachevKondoBook}). The aim of this paper is to give a positive answer to this classical question.

\begin{theorem} \label{thm:main.theorem.introduction}
Every classical nodal Enriques surface is a Reye congruence.
\end{theorem}

In the introduction of \cite{Dolgachev.Reider}, a proof of \Cref{thm:main.theorem.introduction} is claimed to be contained in \cite{DolgachevKondoBook}, which, in its current version, contains only weaker statements (see \cite[Theorem 7.9.1]{DolgachevKondoBook}). 

A linear system of projective dimension \(3\) is called a \emph{web}. 
Given a sufficiently general (more precisely, \emph{regular}, see \cite[Definition 2.1.2]{Cossec:Reye}) web~\(W\) of quadrics in~\(\IP^3\), classical geometers considered the smooth surface \(R(W)\) in \(\IG(1,3)\) of order \(7\) and class \(3\) defined as the set of lines contained in at least two distinct quadrics of~\(W\). Such surfaces, which here we call \emph{classical Reye congruences}, were originally introduced by Darboux \cite{darboux} and then studied by Reye~\cite[Sechzehnter Vortrag]{Reye} and, later, by Fano~\cite{Fano}, who realized that they are examples of Enriques surfaces.
Note that the image of the composition with the Plücker embedding \(R(W) \hookrightarrow \IG(1,3) \hookrightarrow \IP^5\) is a surface of degree~\(10\), so classical Reye congruences are Reye congruences in the sense defined above. 

\begin{remark} \label{rmk:R(W)}
A classical nodal Enriques surface \(X\) that is isomorphic to a classical Reye congruence admits an ample Fano polarization, so it must have non-degeneracy \(10\) (see \Cref{def:non-degeneracy} and \Cref{lem:isotropic.elements.in.10.sequence}). Since there are nodal Enriques surfaces of non-degeneracy smaller than~\(10\), even over \(\IC\) (see, e.g., \cite[Example 4.13]{MMVnd3}), not every nodal Enriques surface is isomorphic to~\(R(W)\) for a suitable~\(W\). 
In \cite[p.~40]{Dolgachev.Reider}, the authors claimed that every nodal Enriques surface of non-degeneracy \(10\) is isomorphic to a classical Reye congruence. It turns out that this statement fails for Enriques surfaces of Kond\={o}'s type~\(\mathrm{VII}\) (see \cite{Kondo:Enriques.finite.aut}), as we explain in \Cref{sec:counterexample}.

Nevertheless, we show in \Cref{rmk:relationtoReye} that \Cref{thm:main.theorem.introduction} implies that every nodal Enriques surface in characteristic different from \(2\) arises from a not necessarily regular web of quadrics via the Reye construction. See \Cref{rmk:relatontoReye2} for a discussion of the case of characteristic \(2\).
\end{remark}

Our strategy of proof for \Cref{thm:main.theorem.introduction} is as follows.
Given a Fano polarization \(H\), the induced morphism \(\varphi_H\colon X \to \IP^5\) factors through a smooth quadric hypersurface if and only if there is a \emph{Reye bundle} associated to~\(H\) (\Cref{def:Reye.bundle}).
In \Cref{lem:existence.Reye.bundle}, we give a necessary and sufficient criterion for the existence of a Reye bundle~\(\cV\) for~\(H\) and we prove in \Cref{thm:main_criterion} that such a Reye bundle~\(\cV\) maps \(X\) birationally to a surface in \(\mathbb{G}(1,3)\). 
This criterion is a generalization of a result of Cossec (see \cite[Theorem~1]{Dolgachev.Reider} and \cite[Section~4]{conte.verra}) to the case of not necessarily ample Fano polarizations. 
Then, in \Cref{thm:special.3-sequence}, we show that every nodal Enriques surface~\(X\) of non-degeneracy at least \(3\) admits a Fano polarization that satisfies our criterion. 
More precisely, we produce a Reye bundle from any special \(3\)-sequence on~\(X\) (see \Cref{rk:geometric.intrepretation} for a geometric interpretation of the construction). 
By \cite[Theorem 1.3]{MMV:extraspecial}, there are only three families of Enriques surfaces of non-degeneracy less than~\(3\), and these only exist in characteristic~\(2\). For these three families, we explicitly construct a Reye bundle in \Cref{sec:extraspecial}.
Finally, in \Cref{sec:counterexample}, we give an example of a nodal Enriques surface of non-degeneracy~\(10\) which is not isomorphic to a classical Reye congruence.

\begin{remark}
We remark that the Reye bundles \(\cV\) constructed in this article are examples of \emph{extremal} vector bundles of rank \(2\), as studied for example in \cite{Brivio}, \cite{Dolgachev.Reider}, \cite{Kim:exc.bundles.nodal}, \cite{Naie:special} and \cite{Zube:exc.vec.bundles}.
\end{remark}

\begin{acknowledgments*}
We thank Igor Dolgachev for helpful comments on a first draft of this article.
\end{acknowledgments*}

\section{Preliminaries} \label{sec:preliminaries}

In this section, we fix our notation and prove a few preliminary results. 
After recalling the definition of \(\Phi\)-function, we review the relationship between Fano polarizations and \(10\)-sequences. 
Moreover, we introduce the notion of negative definite divisor (\Cref{def:neg.def}).

We fix throughout an Enriques surface~\(X\), that is, a smooth and proper surface over an algebraically closed field with numerically trivial canonical class \(K_X\) and \(b_2(X) = 10\). 
We denote by \(\Num(X)\) its numerical lattice. Linear and numerical equivalence are denoted by~\(\sim\) and~\(\equiv\), respectively.
Given a divisor~\(D\), we write \(h^i(D)\) for the dimension of \(H^i(X,\cO_X(D))\). 

A \emph{half-fiber} is a connected, nef divisor \(F\) with \(F^2 = 0\) and \(h^0(F) = 1\). Every Enriques surface carries a half-fiber (see, e.g., \cite[Corollary 2.3.4]{CossecDolgachevLiedtke}). 
For a big and nef divisor \(H\), that is, a nef divisor \(H\) with \(H^2 > 0\), we define (cf. \cite[Lemma~2.4.10]{CossecDolgachevLiedtke})
\[
    \Phi(H) \coloneqq \min\{H.F \mid \text{\(F\) half-fiber on~\(X\)}\}.
\]
We start by recalling an important general result related to Enriques' Reducibility Lemma.

\begin{lemma}[{\cite[Theorem~3.2.1]{Cossec.Dolgachev}} or {\cite[Theorem~2.3.3]{CossecDolgachevLiedtke}}] \label{lem:reducibility}
If \(D\) is an effective divisor with \(D^2 \geq 0\), then there exist non-negative integers \(a_i\) and \((-2)\)-curves \(R_i\) such that \(D \sim D' + \sum_i a_i R_i,\) where \(D'\) is the unique effective nef divisor class in the Weyl orbit of~\(D\). In particular, \(D^2 = D'^2\) and if \(D \neq 0\), then \(D' \neq 0\).
\end{lemma}

\begin{lemma} \label{lem: bigandnefalwayscontainshalffiber}
Given an effective nef divisor \(D\), the following hold:
\begin{enumerate}
    \item If \(D^2 = 0\), then there exists a half-fiber \(F\) and a positive integer \(n\) such that \(D \equiv nF\).
    \item If \(D^2 > 0\) and \(F\) is a half-fiber with \(F.D = \Phi(D)\), then \(h^0(D-F) \neq 0\) and \((D-F)^2 \geq 0\).
\end{enumerate}
\end{lemma}
\begin{proof}
The first claim is \cite[Corollary~2.2.9]{CossecDolgachevLiedtke}. Suppose now that \(D^2 > 0\).
By \cite[Proposition~2.4.11]{CossecDolgachevLiedtke}, we have \(\Phi(D)^2 \leq D^2\). In particular, since \(D^2 \geq 2\), we have \((D-F)^2 = D^2 - 2\Phi(D) \geq 0\), as claimed. 
The Hodge index theorem implies \(\Phi(D) > 0\), which in turn implies \(D.(K_X - D + F) < -(D-F)^2 \leq 0\). Therefore, we have \(h^2(D-F) = h^0(K_X - D + F) = 0\) by Serre duality, whence \(h^0(D-F) \neq 0\) by Riemann--Roch.
\end{proof}

The following inequalities on \(\Phi(H)\) will often be useful.

\begin{proposition} \label{lem:Igors.trick}
Let \(H\) be a big and nef divisor. 
\begin{enumerate}
\item If \(D\) is an effective divisor with \(D^2 \geq 0\), then \(H.D \geq \Phi(H)\).
\item If \(D\) is a divisor with \(h^0(D) \geq 2\), then \(H.D \geq 2\Phi(H)\).
\end{enumerate}
\end{proposition}
\begin{proof}
For Claim (1), write \(D \sim D' + \sum_i a_iR_i\) as in \Cref{lem:reducibility}. By \Cref{lem: bigandnefalwayscontainshalffiber}, there exist a half-fiber \(F\) and an effective divisor \(D''\) with \(D' \equiv F + D''\). Then, we have that \(H.D \geq H.D' \geq H.F \geq \Phi(H)\).

For Claim (2), write \(D = D' + D''\), where \(D'\) is the movable part and \(D''\) the fixed part. Since \(h^0(D) = h^0(D')\) and \(H.D \geq H.D'\) because \(H\) is nef, we can assume without loss of generality that \(D = D'\). In particular, \(D\) is an effective nef divisor with \(D^2 \geq 0\). 

If \(D^2 = 0\), then \(D \equiv nF\) for some half-fiber \(F\) and some \(n \geq 1\) by \Cref{lem: bigandnefalwayscontainshalffiber}. Since \(h^0(D) \geq 2\), we have \(n \geq 2\), so \(H.D \geq 2H.F \geq 2\Phi(H)\). 
If \(D^2 > 0\), pick a half-fiber \(F\) with \(F.D = \Phi(D)\). Then, \(h^0(D-F) \neq 0\) and \((D-F)^2 \geq 0\) by \Cref{lem: bigandnefalwayscontainshalffiber}. Hence, by Claim (1), we have \(H.(D-F) \geq \Phi(H)\), and so \(H.D = H.F + H.(D-F) \geq 2\Phi(H).\)
\end{proof}

\begin{definition} \label{def:Fano.polarization}
A \emph{Fano polarization} is a nef divisor class \(H\) with \(H^2 = 10\) and \(\Phi(H) = 3\).
\end{definition}

\begin{definition} 
A \emph{\(c\)-degenerate (canonical, isotropic) \(10\)-sequence}, or just a \emph{\(10\)-sequence}, is a \(10\)-tuple of divisors \((E_1,\ldots,E_{10})\) of the form 
\[
    \big(F_1,F_1 + R_{1,1}, \ldots, F_1 + \sum_{j=1}^{r_1} R_{1,j}, \ldots, F_c, F_c + R_{c,1}, \ldots, F_{c} + \sum_{j=1}^{r_c} R_{c,j}\big),
\]
where the \(F_i\) are half-fibers of genus one fibrations on~\(X\) and the \(R_{i,j}\) are \((-2)\)-curves satisfying the following conditions:
\begin{itemize}
    \item \(F_i.F_j = 1 - \delta_{ij}\);
    \item \(R_{i,j}.R_{i,j+1} = 1\);
    \item \(R_{i,j}.R_{k,l} = 0\) unless \((k,l) = (i,j)\) or~\((k,l)=(i,j\pm 1)\);
    \item \(F_i.R_{i,1} = 1\) and \(F_i.R_{k,l} = 0\) if \((k,l) \neq (i,1)\).
\end{itemize}
The divisor \(T_i = \sum_{j = 1}^{r_i}R_{i,j}\) is called the \emph{tail} of \(F_i\) in the \(10\)-sequence.
\end{definition}

\begin{definition} \label{def:non-degeneracy}
The \emph{(maximal) non-degeneracy} of \(X\) is the maximal \(c \in \{1,\ldots,10\}\) for which there exists a \(c\)-degenerate sequence on~\(X\).
\end{definition}

The notions of Fano polarization and \(10\)-sequence are intimately connected. 
Indeed, for each \(10\)-sequence \((E_1,\ldots,E_{10})\), the divisors \(E_i\) satisfy \(E_i.E_j = 1-\delta_{ij}\). Hence, the class of \(\sum_{i=1}^{10} E_i\) is always \(3\)-divisible in \(\Num(X)\). 
More precisely, there exists a unique divisor \(H\) up to linear equivalence such that \(3H \sim \sum_{i = 1}^{10} E_i\). One can easily prove that \(H\) is a Fano polarization. 

Conversely, for each Fano polarization \(H\), it is shown in \cite[Lemma 3.5.2]{CossecDolgachevLiedtke} that there exists a \(10\)-sequence \((E_1,\ldots,E_{10})\) satisfying \(3H \equiv \sum_{i = 1}^{10} E_i\). This \(10\)-sequence is uniquely determined up to reordering and up to exchanging any half-fiber \(F_i\) appearing in it with a numerically equivalent half-fiber \(F_i'\). With a slight abuse of language, we will call it \emph{the} \(10\)-sequence associated to~\(H\). 

\Cref{lem:isotropic.elements.in.10.sequence} characterizes all divisors \(E_i\), and, in particular, the half-fibers \(F_i\) and the \((-2)\)-curves \(R_{i,j}\) appearing in the \(10\)-sequence.

\begin{lemma} \label{lem:isotropic.elements.in.10.sequence}
Given a Fano polarization \(H\), the following hold:
\begin{enumerate}
\item If \(R\) is a \((-2)\)-curve, then \(R\) appears in the \(10\)-sequence associated to \(H\) if and only if \(H.R = 0\).
\item If \(E\) is an effective, isotropic divisor (e.g., a half-fiber), then \(E\) appears in the \(10\)-sequence associated to \(H\) if and only if \(H.E = 3\). In this case, \(h^0(E) = 1\) and \(h^1(E) = 0\).
\end{enumerate}
\end{lemma}
\begin{proof}
Let \((E_1,\ldots,E_{10})\) be the \(10\)-sequence associated to~\(H\).

For Claim \((1)\), observe first that \(H.R_{i,j} = 0\) for each \(R_{i,j}\) in the \(10\)-sequence. On the other hand, any \((-2)\)-curve \(R\) with \(H.R = 0\) cannot satisfy \(R.E_i = 0\) for all \(i\), because the \(E_i\) generate \(\Num(X)\) over \(\IQ\) and \(R \not\equiv 0\). Hence, \(R.E_i < 0\) for some \(i\), so \(R\) appears in the \(10\)-sequence.

For Claim \((2)\), observe that \(3H.E_j = \sum_{i = 1}^{10}E_i.E_j = 9\), hence \(H.E_j=3\). 
Conversely, let \(E\) be any effective isotropic divisor with \(H.E=3\), and assume by contradiction that \(E\) is not numerically equivalent to any of the~\(E_i\). We have \(E.E_i\ne 0\), since for each \(i\in \{1,\ldots,10\}\) the sublattice \(E_i^\perp \subseteq \Num(X)\) is negative semidefinite, with \(1\)-dimensional kernel generated by \(E_i\).
We claim that \(E.E_i> 0\) for each \(i\in \{1,\ldots,10\}\). If instead \(E.E_i<0\), then \((E-E_i)^2>0\) and \(h^0(E-E_i)+h^2(E-E_i)=h^0(E-E_i)+h^0(E_i-E+K_X)\ge 2\) by Riemann-Roch. Since \(E-E_i\) and \(E_i-E+K_X\) cannot both be linearly equivalent to an effective divisor, we infer that either \(h^0(E-E_i)\ge 2\) or \(h^0(E_i-E+K_X)\ge 2\),  contradicting \Cref{lem:Igors.trick} in both cases.
Thus \(E.E_i\ge 1\) for each \(i\in \{1,\ldots,10\}\), and so \(H.E \ge \frac{10}{3}>3\), a contradiction.

Finally, if \(H.E = 3\), then \(h^0(E) = 1\) by \Cref{lem:Igors.trick} and \(h^2(E) = 0\) by Serre duality; therefore, \(h^1(E) = 0\) by Riemann--Roch.
\end{proof}

\begin{lemma} \label{lem:H.R=1.or.2}
Given a Fano polarization \(H\), let \(E_i\) be the divisors appearing in its associated \(10\)-sequence, and let \(R\) be a \((-2)\)-curve.
\begin{enumerate}
    \item If \(H.R = 1\), then, up to reordering, \(R.E_i = 1\) for \(i \in \{1,2,3\}\) and \(R.E_i = 0\) else.
    \item If \(H.R = 2\), then, up to reordering, \(R.E_i = 1\) for \(i \in \{1,\ldots,6\}\) and \(R.E_i = 0\) else.
\end{enumerate}
\end{lemma}
\begin{proof}
In both cases, \(R\) does not appear in the \(10\)-sequence because \(H.R \neq 0\). Up to reordering, we can suppose that \(R.E_1 \geq \dots \geq R.E_{10} \geq 0\). 

If \(H.R = 1\), then \(\sum R.E_i = 3\), which implies that the vector \(v \in \IZ^{10}\) with coordinates \(v_i = R.E_i\) is equal to either \((3,0,\ldots,0)\), \((2,1,0,\ldots,0)\) or \((1,1,1,0,\ldots,0)\). Since the divisors~\(E_i\) generate \(\Num(X)\) over ~\(\IQ\), the intersection numbers \(R.E_i\) determine \(R^2\) uniquely. 
More precisely, the dual vector of \(E_i\) is \(\frac{1}{3}H - E_i\), so \(R^2 = (\frac{1}{3}\sum v_iH - \sum v_iE_i)^2\). Thus, \(R^2 = -2\) if and only if \(v = (1,1,1,0,\ldots,0)\), which yields the claim. 
An analogous argument applies to the case \(H.R = 2\).
\end{proof}

\begin{proposition} \label{prop:nefness.H-F}
Let \(H\) be a Fano polarization, let \(F\) be a half-fiber, and suppose that \(E = F +R_{1}+\ldots+R_{n}\) is a divisor appearing in the \(10\)-sequence associated to~\(H\). Then, the divisor \(H-E\) is nef if and only if \(n\) is the length of the tail of~\(F\). In this case, \(h^0(H - E) = 3\).
\end{proposition}
\begin{proof}
If \(R_{n+1}\) exists, then \(H-E\) is not nef, since \((H-E).R_{n+1}<0\). 

For the converse implication, let \(R\) be a \((-2)\)-curve such that \((H-E).R < 0\). Since \(E\) appears in the \(10\)-sequence associated to~\(H\), we have \(H.E = 3\) by \Cref{lem:isotropic.elements.in.10.sequence}. We observe that \(h^2(H-E) = h^0(E - H+K_X) = 0\) because \(H.(E-H-K_X) = -7 < 0\), and that \((H-E)^2 = 4\). By Riemann--Roch, \(h^0(H-E) \geq 3\).
The \((-2)\)-curve \(R\) is a fixed component of \(|H-E|\), so \(H - E - R\) is effective and \(h^0(H-E-R) \geq 3\) as well. By \Cref{lem:Igors.trick}, we have \(H.(H-E-R) \geq 2\Phi(H) = 6\), whence \(H.R \leq 1\). Since \(H\) is nef, either \(H.R = 0\) or \(H.R = 1\).

Assume that \(H.R=1\). By \Cref{lem:H.R=1.or.2}, we have \(R.E \leq 1\), which implies that \((H-E).R \geq 1-1 = 0\), a contradiction. Therefore, \(H.R = 0\), i.e., \(R\) is a \((-2)\)-curve in the \(10\)-sequence. 
It follows from \(E.R = -(H-E).R > 0\) that \(R = R_{n+1}\).

When \(H - E\) is nef, it holds that also \(h^1(H - E) = 0\) by \cite[Theorem~2.1.15]{CossecDolgachevLiedtke}. Thus, \(h^0(H - E) = 3\) by Riemann--Roch.
\end{proof}

\begin{definition} \label{def:neg.def}
An effective divisor \(C\) is \emph{negative definite} if \(D^2 < 0\) for every effective divisor \(D \neq 0\) such that \(C - D\) is effective.
\end{definition}

\begin{remark} \label{rk:neg.def}
Any negative definite divisor \(C\) is \emph{rigid}, i.e., \(h^0(C) = 1\), and its components are smooth rational curves. Indeed, if either of the two conditions were violated, we could use \Cref{lem:reducibility} to write \(C = N + D\) for \(D\) effective and \(N \ne 0\) with \(N^2 \ge 0\), contradicting \Cref{def:neg.def}.

Note, however, that negative definite divisors are not necessarily nodal cycles; in other words, their support might not be a Dynkin diagram. For instance, the effective divisor pictured in \eqref{eq:case.E8.extra.special} is negative definite, but its support spans the whole hyperbolic lattice \(\Num(X)\).
\end{remark}

\begin{lemma} \label{lem:negativedefinite} 
Let \(C\) be an effective divisor on a classical Enriques surface \(X\). If \(C\) is negative definite, then \(|C+K_X|=\emptyset\).
\end{lemma}
\begin{proof}
By \Cref{rk:neg.def}, the linear system \(|C|\) contains a unique divisor, namely \(C\) itself.
We prove the statement by induction on the number \(n\) of components of \(C\). If \(n = 0\), the claim is clear, since \(K_X \neq 0\).

For the induction step, we may assume that \(C \neq 0\), hence \(C^2 < 0\). In particular, \(C\) is not nef, and we can choose a \((-2)\)-curve \(R\) on~\(X\) with \(C.R < 0\). If \(C' \in |C + K_X|\), then \(C'.R <0\), so \(R\) is a component of both \(C\) and \(C'\).
Now, \(C-R\) is negative definite and \(C' - R \in |C + K_X - R|\), a contradiction to the induction hypothesis.
\end{proof}

\begin{proposition} \label{prop:H-2F.is.negative.definite}
Let \(H\) be a Fano polarization, \(F\) a half-fiber appearing in its associated \(10\)-sequence and \(T\) its tail. If \(H-2(F+T)\) is numerically equivalent to an effective divisor \(C\), then \(C\) is negative definite.
\end{proposition}
\begin{proof}
Assume by contradiction that \(C\) is not negative definite, so that, by \Cref{lem:reducibility} applied to a subdivisor of non-negative square, there exists a decomposition \(C=N+D\) with \(N\ne 0\) nef and \(D\) effective. Let \(E = F + T\). 

By \Cref{prop:nefness.H-F}, \(H - E\) is nef.
We have \((H - E).N \leq (H-E).C = (H-E).(H-2E) = 1\), which also implies \(E.N = (H-E - C).N \leq (H-E).N \leq 1\). Combining the two inequalities, we obtain \(H.N = (H - E).N+ E.N \le 2\). 

However, by \Cref{lem:Igors.trick}, \(H.N \geq \Phi(H) = 3\). This contradiction shows that \(C\) must be negative definite.
\end{proof}

\section{Reye bundles} \label{sec:Reye.bundles}

Throughout this section, we fix a classical Enriques surface \(X\), that is, an Enriques surface such that the canonical class \(K_X\) is not \(0\), and a Fano polarization \(H\) on~\(X\). We keep denoting by \(F_i\) and \(T_i\) the half-fibers appearing in the \(10\)-sequence associated to~\(H\) and their tails, respectively. 

The complete linear system \(|H|\) induces a morphism \(\varphi_H \colon X \to \IP^5\) that maps \(X\) birationally onto a normal surface of degree~\(10\) with at worst rational double points (see, e.g., \cite[Theorem~2.4.16]{CossecDolgachevLiedtke}).

\begin{definition} \label{def:Reye.polarization}
A Fano polarization \(H\) is a \emph{Fano--Reye polarization} if the image of the induced morphism \(\varphi_H \colon X \to \IP^5\) is contained in a smooth quadric hypersurface. The Enriques surface \(X\) is called a \emph{Reye congruence} if \(X\) admits a Fano--Reye polarization.
\end{definition}

The aim of this section, which we achieve in~\Cref{thm:main_criterion}, is to find a criterion for \(H\) to be a Fano--Reye polarization.

\begin{definition} \label{def:Reye.bundle}
A \emph{Reye bundle} for \(H\) is a locally free sheaf \(\cV\) of rank~\(2\) on~\(X\) which fits into a non-split exact sequence of the form 
\begin{equation} \label{eq:Reye.bundle}
    0 \longrightarrow \cO_X(F_i + T_i) \longrightarrow \cV \longrightarrow \cO_X(H - F_i - T_i) \longrightarrow 0
\end{equation}
for some half-fiber \(F_i\) appearing with tail \(T_i\) in \(10\)-sequence associated to~\(H\).
\end{definition}

\begin{lemma} \label{lem:existence.Reye.bundle}
A necessary and sufficient condition for the existence of a Reye bundle~\(\mathcal{V}\) for~\(H\) is that
\[
    |H - 2(F_i + T_i) + K_X| \neq \emptyset.
\]
for some \(i\). Moreover, if \(\cV\) exists, then it is uniquely determined by the fact that it fits into a non-split short exact sequence as in \eqref{eq:Reye.bundle}.
\end{lemma} 
\begin{proof}
If we denote the divisor \(H - 2(F_i + T_i) + K_X\) by \(C\), then we have that
\[
    \Ext^1(\cO_X(H - F_i - T_i), \cO_X(F_i + T_i)) \cong H^1(X, \cO_X(C))^\vee.
\]
Thus, the existence of the non-split exact sequence in \eqref{eq:Reye.bundle} is equivalent to \(h^1(C) \neq 0\). We have \(h^2(C) = h^0(K_X - C) = 0\) because \(H.(K_X - C) = -4 < 0\). Since \(C^2 = -2\), Riemann--Roch implies that \(h^0(C) = h^1(C)\). 
In particular, \(h^1(C) \neq 0\) if and only if \(|C| \neq \emptyset\).

By \Cref{lem:Igors.trick}, \(H.C = 4 < 2\Phi(H) = 6\) implies that \(h^0(C) \leq 1\), from which we infer the last statement.
\end{proof} 

\begin{lemma} \label{lem:h0.Reye.bundle.4}
A Reye bundle \(\cV\) for \(H\) satisfies \(c_1(\cV) = H\), \(c_2(\cV) = 3\) and \(h^0(X, \cV) = 4\).
\end{lemma}
\begin{proof}
From the sequence in \eqref{eq:Reye.bundle}, we infer that \(c_1(\cV) = (F_i + T_i) + (H-F_i - T_i) = H\) and \(c_2(\cV) = (F_i + T_i).(H - F_i - T_i) = 3\).

By \Cref{lem:isotropic.elements.in.10.sequence}, we have that \(h^0(F_i + T_i) = 1\) and \(h^1(F_i + T_i) = 0\). 
By \Cref{prop:nefness.H-F}, \(h^0(H - F_i - T_i) = 3\).
Therefore, we deduce that \(h^0(X, \cV) = 4\) from the sequence in \eqref{eq:Reye.bundle}.
\end{proof}

On a classical Enriques surface, each genus one fibration admits exactly two half-fibers (see, e.g., \cite[Corollary 2.2.9]{CossecDolgachevLiedtke}). In the following, we denote by \(F_i'\) the other half-fiber such that \(2F_i \sim 2F_i'\). 

\begin{proposition} \label{prop:Reye.bundle.F_i'}
A Reye bundle \(\cV\) fitting in \eqref{eq:Reye.bundle} also fits into a non-split short exact sequence of the form
\begin{equation} \label{eq:Reye.bundle.F_i'}
    0 \longrightarrow \cO_X(F_i' + T_i) \longrightarrow \cV \longrightarrow \cO_X(H - F_i' - T_i) \longrightarrow 0.
\end{equation}
\end{proposition}
\begin{proof}
By \Cref{lem:existence.Reye.bundle}, the divisor \(H - 2(F_i + T_i) + K_X\) is linearly equivalent to an effective divisor~\(C\). Tensoring the sequence in \eqref{eq:Reye.bundle} with \(\cO_X(-F_i' - T_i)\), we obtain
\[
    0 \longrightarrow \cO_X(K_X) \longrightarrow \cV(-F_i' - T_i) \longrightarrow \cO_X(C) \longrightarrow 0.
\]
As \(h^0(K_X) = h^1(K_X) = 0\), we see that \(h^0(X,\cV(-F_i' - T_i)) \neq 0 \), which implies that there is a non-zero morphism \(\cO_X(F_i' + T_i) \rightarrow \cV\). 

We consider the induced saturated short exact sequence
\begin{equation} \label{eq:saturated.sequence}
    0 \longrightarrow \cO_X(F_i'+ T_i + D) \longrightarrow \cV \longrightarrow \cI_Z \otimes \cO_X(H - F_i' - T_i - D) \longrightarrow 0,
\end{equation}
which is constructed as follows.
If \(\pi\colon \cV \to \cV/ \cO_X(F_i' + T_i)\) denotes the induced quotient map,
then 
\(
    \pi^{-1}(\Tors(\cV/\cO_X(F_i' + T_i))) \cong \cO_X(F_i' + T_i + D)
\)
for some effective divisor \(D\).
The quotient \(\cV/\cO_X(F_i' + T_i + D)\) is a torsion-free sheaf of rank~\(1\), hence of the form \(\cI_Z \otimes \cO_X(D')\) for some closed subscheme \(Z\) of dimension \(0\) and some divisor~\(D'\).
From \(c_1(\cV) = H\) (\Cref{lem:h0.Reye.bundle.4}) and \(c_1(\cI_Z) = 0\), we deduce that \(D' \sim H - F_i' - T_i - D\). 

Computing \(c_2(\cV)\) using the sequence in \eqref{eq:saturated.sequence}, we obtain 
\begin{align*}
    c_2(\cV) & = (F_i' + T_i + D).(H - F_i' - T_i  - D) + \length(\cO_Z) = \\
        & = 3 + D.(H - 2F_i - 2T_i - D) + \length(\cO_Z).
\end{align*}
Since \(c_2(\cV) = 3\) by \Cref{lem:h0.Reye.bundle.4}, we infer by comparison that
\begin{equation} \label{eq:comparison.c2(V)}
    D.(C - D) = -\length(\cO_Z) \leq 0.
\end{equation}

Observe that \[\Hom(\cO_X(F_i' + T_i + D),\cO_X(F_i + T_i)) \cong H^0(X, \cO_X(F_i - F_i' - D))=0,\] since \(h^0(F_i - F_i' - D) \leq h^0(F_i - F_i') = h^0(K_X) = 0\). 
Hence, the composition of the two maps from \eqref{eq:Reye.bundle} and \eqref{eq:saturated.sequence}
\[
    \cO_X(F_i' + T_i + D) \longhookrightarrow \cV \longrightarrow \cO_X(H-F_i-T_i)
\]
cannot be zero, otherwise we would have a non-zero morphism from \(\cO_X(F_i' + T_i +D)\) to \(\Ker(\cV \rightarrow \cO_X(H - F_i - T_i)) = \cO_X(F_i + T_i)\). 
Thus,
\[
    H^0(X,\cO_X(C-D)) \cong \Hom(\cO_X(F_i' + T_i +D), \cO_X(H - F_i - T_i)) \neq 0,
\]
i.e., \(C-D\) is linearly equivalent to an effective divisor. Since \(C\) is rigid and \(D\) is effective, \(C - D\) is in fact effective.

We claim that \(D = 0\), using the fact that \(C\) is a negative definite divisor by \Cref{prop:H-2F.is.negative.definite}. 
If we had \(C = D\), then \eqref{eq:comparison.c2(V)} would imply \(\length(\cO_Z) = 0\), i.e., \(Z = \emptyset\), and the saturated sequence in \eqref{eq:saturated.sequence} would yield a non-zero morphism \(\cV \to \cO_X(F_i + T_i)\), contradicting the fact that the sequence in \eqref{eq:Reye.bundle} does not split. 
Thus, it holds that \(C \neq D\); in particular, \((C-D)^2 \leq -2\) by negative definiteness of~\(C\). Using \(C^2 = -2\) and \eqref{eq:comparison.c2(V)}, we obtain \(D^2 \geq  D^2 + 2D.(C-D) = -(C-D)^2-2 \geq 0\). Therefore, \(D = 0\) again by negative definiteness of~\(C\).

Now that we know \(D = 0\), \eqref{eq:comparison.c2(V)} also shows that \(\length(\cO_Z) = 0\), i.e., \(Z = \emptyset\); hence, the Reye bundle \(\cV\) fits into a short exact sequence as in \eqref{eq:Reye.bundle.F_i'}.
It remains to show that this sequence does not split. 

Seeking a contradiction, assume that it does. Then, the splitting gives a non-zero morphism
\(\varphi\colon \cO_X(H - F_i' - T_i) \to \cV\), whose composition with the morphism \(\cV \rightarrow \cO_X(H - F_i - T_i)\) from~\eqref{eq:Reye.bundle} must be zero, for otherwise it would yield a global section of the canonical sheaf \(\omega_X\). Hence, \(\varphi\) factors through a non-zero morphism to \(\Ker(\cV \rightarrow \cO_X(H - F_i - T_i)) = \cO_X(F_i + T_i)\). Thus,
\[
    H^0(X,\cO_X(-C)) \cong \Hom(\cO_X(H - F_i' - T_i),\cO_X(F_i + T_i)) \neq 0,
\]
contradicting \(H.(-C) = -4 < 0\). 
\end{proof}

An argument analogous to the one in the proof of \Cref{prop:Reye.bundle.F_i'} also shows the following.

\begin{proposition} \label{prop:Reye.bundle.properties}
There is an isomorphism \(\cV(K_X)\cong \cV\), and \(\cV\) fits into non-split exact sequences of the form
\begin{align*} 
     0 \longrightarrow \cO_X(F_i + T_i) &\longrightarrow \cV \longrightarrow \cO_X(H - F_i - T_i) \longrightarrow 0 \\
     0 \longrightarrow \cO_X(F_i' + T_i) &\longrightarrow \cV \longrightarrow \cO_X(H - F_i' - T_i) \longrightarrow 0
 \end{align*}
for all \(i\). In particular:
\begin{enumerate}
    \item It holds \(|H - 2(F_i + T_i) + K_X| \neq \emptyset\) for some \(i\) if and only if \(|H - 2(F_i + T_i) + K_X| \neq \emptyset\) for all \(i\).
    \item If \(\cV\) exists, it is uniquely determined by the Fano polarization \(H\), and we call it \emph{the} Reye bundle associated to~\(H\).
\end{enumerate}
\end{proposition}

\begin{theorem} \label{thm:main_criterion}
On a classical Enriques surface \(X\), every Fano polarization \(H\) which admits a Reye bundle is a Fano--Reye polarization. 
In particular, if \(H - 2(F_i + T_i)\) is numerically equivalent to an effective divisor for some half-fiber \(F_i\) appearing with tail \(T_i\) in the \(10\)-sequence associated to~\(H\), then \(X\) is a Reye congruence.
\end{theorem}
\begin{proof}
By \Cref{prop:Reye.bundle.properties}, we may assume without loss of generality that \(i = 1\).
The second statement follows directly from the first and \Cref{lem:existence.Reye.bundle}: if \(H - 2(F_1 + T_1)\) is numerically equivalent to an effective divisor, then either \(H\) or \(H + K_X\) is a Fano--Reye polarization. From now on, we will assume that \(H - 2(F_1 + T_1) + K_X\) is linearly equivalent to an effective divisor~\(C\). By \Cref{prop:H-2F.is.negative.definite}, \(C\) is negative definite.

Since \(\cV\) satisfies \({\rm det}(\cV) \cong \cO_X(H)\) and \(h^0(X,\cV) = 4\) by \Cref{lem:h0.Reye.bundle.4}, we get an induced rational map \(\varphi_\cV \colon X \dashrightarrow \mathbb{G}(2,H^0(X,\cV))\) to the Grassmannian that fits into the following commutative diagram
\[
    \begin{tikzcd}
        X \arrow[rr,dashed,"\varphi_{\cV}"] \arrow[d,"\varphi_H"']     &           & \IG(2, H^0(X,\cV)) \arrow[d] \\
        \IP(H^0(X, \cO_X(H))^\vee) \arrow[rr, "\IP(\gamma^\vee)"] && \IP((\bigwedge^2 H^0(X, \cV))^\vee)
    \end{tikzcd}
\]
where \(\gamma \colon \bigwedge^2 H^0(X,\cV) \to H^0(X,\cO_X(H))\) is the natural map.

It suffices to show that \(\gamma\) is an isomorphism. Indeed, if this is the case, then the global generation of \(\cO_X(H)\) implies the global generation of \(\cV\), so \(\varphi_\cV \colon X \to \IG(2,H^0(X,\cV))\) is a morphism and the morphism \(\varphi_H \colon X \to \IP(H^0(X, \cO_X(H))^\vee)\cong \IP^5\) induced by \(|H|\) factors through the smooth quadric \(\mathbb{G}(2,H^0(X,\cV))\). In particular, \(H\) is a Fano--Reye polarization.

In the case where the \(10\)-sequence associated to~\(H\) is \(10\)-degenerate, the proof can be found in \cite[Proposition~3.13]{conte.verra}.
In the proof, the authors only use the fact that \(C\) is a negative definite divisor. Since \(F_1.C=3\) and \(H+K_X = 2F_1 + C\), their proposition implies that \(H\) is a Fano--Reye polarization. 
Hence, we may assume in the following that the \(10\)-sequence is degenerate, i.e., there is a non-trivial tail, which we may assume to be \(T_1\) by \Cref{prop:Reye.bundle.properties}. 

In the \(6\)-dimensional vector space \(H^0(X,\cO_X(H))\), we consider the subspace 
\[
    W \coloneqq H^0(X,\cO_X(H - F_1 - T_1)),
\]
which has dimension \(3\) by \Cref{prop:nefness.H-F}. We claim that \(W\) is in the image of \(\gamma\).
To see this, let \(s_0\) be a section of \(\cV\) that vanishes along \(F_1 + T_1\), and extend it to a basis \(s_0,\hdots,s_3 \in H^0(X,\cV)\). 
If \(\gamma(s_0 \wedge s) = 0\) for some \(0 \neq s \in \langle s_1,s_2,s_3 \rangle\), then the saturation of the subsheaf of \(\cV\) spanned by \(s_0\) and \(s\) is locally free of rank \(1\) and is of the form \(\cO_X(F_1 + T_1 + D)\) for some effective divisor~\(D\). 
The same argument as in the proof of \Cref{prop:Reye.bundle.F_i'}, using the sequence in \eqref{eq:Reye.bundle.F_i'} instead of the one in \eqref{eq:Reye.bundle}, shows that such a \(D\) is necessarily \(0\). 
But then, we have that \(s_0,s \in H^0(X,\cO_X(F_1 + T_1))\), which is a \(1\)-dimensional subspace of \(H^0(X,\cV)\), contradicting the linear independence of \(s_0\) and~\(s\).

The same argument shows that the subspace \(W' \coloneqq H^0(X,\cO_X(H - F_1' - T_1))\) is in the image of~\(\gamma\), hence so is \(W + W'\), which is contained in \(H^0(X, \cO_X(H - T_1))\). As \(H\) is base point free (see, e.g., \cite[Proposition 3.1.1]{CossecDolgachevLiedtke}), it holds that \(h^0(H-T_1) \leq 5\). We claim that \(W + W'\) has exactly dimension \(5\), and, thus, coincides with \(H^0(X, \cO_X(H - T_1))\).

Since \(F_1\) and \(F_1'\) are disjoint, a section of \(H\) vanishing along \(F_1 + T_1\) and on \(F_1'+T_1\) also vanishes on \(F_1+F_1'+T_1\), i.e., we have \(W \cap W' \subseteq H^0(X,\cO_X(C))\)
as subspaces of \(H^0(X,\cO_X(H))\). Moreover, \(h^0(C) = 1\) by \Cref{lem:Igors.trick}, since \(C\) is effective by assumption and \(H.C = 4\). Thus, 
\(\dim(W + W') = \dim(W) + \dim(W') - \dim(W \cap W') \geq 5\), as claimed.

Summing up, the image of \(\gamma\) contains the \(5\)-dimensional subspace \(H^0(X,\cO_X(H-T_1))\). In order to prove the surjectivity of \(\gamma\), it suffices to find two sections \(t_0,t_1\in H^0(X,\cV)\) such that the image \(\gamma(t_0\wedge t_1)\) does not vanish on \(T_1\). To this aim, consider the sequence
\[
0 \longrightarrow \cO_{X}(H-2F_1-2T_1) \longrightarrow \cO_X(H-F_1-2T_1) \longrightarrow \cO_{F_1}(H-F_1-2T_1) \longrightarrow 0.
\]
By \Cref{prop:H-2F.is.negative.definite} and \Cref{lem:negativedefinite}, we have that \(h^0(H-2F_1-2T_1) = 0\). 
By Riemann--Roch, also \(h^1(H-2F_1-2T_1) = 0\). 
Since \(T_1 \neq 0\), we have \((H-F_1-2T_1).F_1 = 1\), whence \(h^0(F_1,\cO_{F_1}(H-F_1-2T_1)) = 1\). 
We deduce that \(h^0(H-F_1-2T_1) = 1\). 

If \(r_1\) is the length of the tail \(T_1 = R_{1,1}+\ldots+R_{1,r_1}\), we claim that \(h^0(H-F_1-T_1-R_{1,r_1}) = 1\) as well. 
Indeed, \((H-F_1-T_1-R_{1,r_1}).R_{1,r_1-1}=-1\), so \(R_{1,r_1-1}\) is in the fixed part of \(|H-F_1-T_1-R_{1,r_1}|\) and, recursively, \(h^0(H-F_1-T_1-R_{1,r_1}) = h^0(H-F_1-2T_1) = 1\).

By \Cref{lem:isotropic.elements.in.10.sequence}, we have that \(h^0(F_1 + T_1 - R_{1,r_1}) = 1\) and \(h^1(F_1 + T_1 - R_{1,r_1}) = 0\).
Therefore, the sequence
\[
    0 \longrightarrow \cO_X(F_1+T_1-R_{1,r_1}) \longrightarrow \cV(-R_{1,r_1}) \longrightarrow \cO_X(H-F_1-T_1-R_{1,r_1}) \longrightarrow 0
\]
shows that \(h^0(X,\cV(-R_{1,r_1})) = 2\).
In particular \(h^0(R_{1,r_1},\cV|_{R_{1,r_1}})\ge 2\), so we can find two global sections \(t_0,t_1\) of \(\cV\) such that \(t_0|_{R_{1,r_1}},t_1|_{R_{1,r_1}}\) are linearly independent. Therefore, \(\gamma(t_0 \wedge t_1)\) is a section of \(\cO_X(H)\) that does not vanish on \(R_{1,r_1}\), hence not on \(T_1\), as desired.
\end{proof}

\begin{remark} \label{rmk:relationtoReye}
If the characteristic is not \(2\), we explain here how to connect what we call Reye congruence to the classical construction of Reye following \cite[Section 4]{conte.verra}.
Let \(\pi \colon \widetilde{X} \to X\) be the K3 cover with deck involution \(\sigma\), and let \(\mathcal{V}\) be the Reye bundle associated to \(H\). 
Since the effective divisor in \(|H - 2F_1 - 2T_1 + K_X|\) is negative definite by \Cref{prop:H-2F.is.negative.definite}, it splits under \(\pi\) into two disjoint negative definite configurations \(R^+\) and \(R^-\) with \(\sigma(R^\pm) = R^\mp\). We set \(\mathcal{L}^{\pm} \coloneqq \mathcal{O}_{\widetilde{X}}(\pi^*(F_1+T_1) + R^{\pm})\), so that \(\sigma^* \mathcal{L}^\pm \cong \mathcal{L}^{\mp}\).

One can show that \(\pi^*\cV\cong \mathcal{L}^+\oplus \mathcal{L}^-\). It holds \((\mathcal{L}^\pm)^2=4\), and the pair \((\mathcal{L}^+,\mathcal{L}^-)\) induces a morphism \(\psi \colon \widetilde{X}\to \IP^3\times \IP^3\) that fits in the following commutative diagram
\[\begin{tikzcd}
    \widetilde X \arrow{d}[swap]{\pi} \arrow{r}{\psi} & \mathbb{P}^3 \times \mathbb{P}^3 \arrow[dashed]{d}{\lambda} \\
    X \arrow{r}{\varphi_{\cV}}  & \mathbb{G}(1,3),
\end{tikzcd}\]
where the rational map \(\lambda \colon \IP^3\times \IP^3 \dashrightarrow \IG(1,3)\) sends two points of \(\IP^3\) to the line through them and the involution induced by $\sigma$ on $\psi(\widetilde{X})$ coincides with the automorphism induced by swapping the two factors of $\mathbb{P}^3 \times \mathbb{P}^3$. The image \(\psi(\widetilde{X})\) is a surface of degree \((\pi^*H)^2 = 20\) with only rational double points, and it is the complete intersection of four symmetric \((1,1)\)-divisors, corresponding to the $4$-dimensional kernel of the multiplication map $H^0(\widetilde{X},\mathcal{L}^+) \otimes H^0(\widetilde{X},\mathcal{L}^{-}) \to H^0(\widetilde{X}, \pi^* \mathcal{O}_X(H))$. Moreover, \(\psi(\widetilde{X})\) is disjoint from the diagonal \(\Delta \subseteq \IP^3\times \IP^3\), since \(\sigma \) is fixed point free and the divisors contracted by \(\psi\) do not admit fixed point free involutions, so the map \(\lambda\) is defined on \(\psi(\widetilde{X})\).

By interpreting the four \((1,1)\)-divisors as symmetric bilinear forms \(b_i\), their associated quadratic forms \(q_i\) span a web \(W = \langle q_1,\hdots,q_4\rangle\) of quadrics in~\(\IP^3\). 
Now, given a line \(\ell\) in~\(\IP^3\), a computation in bilinear algebra shows the following: there exist distinct points \(x\) and \(y\) on \(\ell\) such that \(b_i(x,y) = 0\) for all \(i\) if and only if there exists a \(2\)-dimensional subspace of \(\langle q_1,\hdots,q_4\rangle\) that vanishes along~\(\ell\). 
In other words, \(\varphi_{\mathcal{V}}(X) = \lambda(\psi(\widetilde{X}))\) is the surface \(R(W)\) associated to the web \(W\) and \(\psi(\widetilde{X})\) is the polar base locus \({\rm PB}(W)\) of the web \(W\) (see \cite[Section 7.2]{DolgachevKondoBook}).
\end{remark}

\begin{remark} \label{rmk:relatontoReye2}
The classical Reye construction has been extended to characteristic~\(2\) in \cite[Sections 7.3 and 7.6]{DolgachevKondoBook}: as in characteristic different from~\(2\), the classical Reye congruence associated to a general web \(W = \langle q_1,\hdots,q_4\rangle\) of quadrics is a classical Enriques surface \(X\). 
For general \(W\), the bilinear form \(b_q\) associated to a quadratic form \(q \in W\) by polarization is non-zero, and the pair \((b_q,q)\) determines a section of  \(\mathcal{O}_{\mathbb{T}(\mathbb{P}^3)}(2)\), where  \(\mathbb{T}(\mathbb{P}^3)\) is the completion of the tangent bundle of~\(\mathbb{P}^3\).
The polar base locus \({\rm PB}(W)\) of \(W\) is defined to be the intersection of the zero loci \(Z(b_{q_i},q_i)\).
The finite part of the Stein factorization of the natural map \(\lambda \colon {\rm PB}(W) \to \mathbb{G}(1,3)\) that sends a pair of a point \(x\) and a tangent direction \(t_x\) to the line through \(x\) with tangent direction \(t_x\) is a \(\mu_2\)-torsor over \(X\) that coincides with its canonical cover.

Therefore, in analogy with \Cref{rmk:relationtoReye}, it seems natural to expect that every classical Enriques surface with a Reye bundle \(\mathcal{V}\) in characteristic \(2\) fits into a diagram of the form
\[\begin{tikzcd}
    \widetilde X \arrow{d}[swap]{\pi} \arrow[dashed]{r}{\psi} & \mathbb{T}(\mathbb{P}^3) \arrow[dashed]{d}{\lambda} \\
    X \arrow{r}{\varphi_{\cV}}  & \mathbb{G}(1,3),
\end{tikzcd}\]
where \(\pi \colon \widetilde{X} \to X\) is the (always singular, and possibly non-normal) canonical cover, \(\varphi_{\mathcal{V}}(X)\) is the Reye congruence associated to a web~\(W\), and the closure of \(\psi(\widetilde X)\) is the polar base locus of~\(W\). It would be interesting to investigate this.
\end{remark}

\section{Fano--Reye polarizations on non-extra-special Enriques surfaces} \label{sec:non-extra-special}

In this section, we prove \Cref{thm:main.theorem.introduction} for Enriques surfaces of non-degeneracy at least \(3\). Enriques surfaces of non-degeneracy strictly smaller than \(3\) are called \emph{extra-special}, and they exist only in characteristic \(2\) (see \cite[Corollary~1.6]{MMV:extraspecial}). We will analyze extra-special Enriques surfaces in detail in the next section.

Throughout this section, \(H\) denotes a Fano polarization, and \((F_1, F_2, F_3)\) denotes a non-degenerate \(3\)-sequence, i.e., a triple of half-fibers with \(F_i.F_j = 1 - \delta_{ij}\), such that \(H.F_1 = H.F_2 = H.F_3 = 3\).
By~\Cref{lem:isotropic.elements.in.10.sequence}, the half-fiber \(F_i\) appears in the \(10\)-sequence associated to~\(H\), and we call \(T_i\) its tail, which may well be empty. In any case, the divisors \(F_i + T_i\) appear in the \(10\)-sequence; in particular, they satisfy \(H.(F_i+T_i) = 3\), \((F_i+T_i)^2 = 0\) and \((F_i+T_i).(F_j + T_j) = 1\) for \(i \neq j\).

\begin{proposition} \label{prop:nontrivial.tilt}
If \(H - F_1 - F_2 - F_3\) is not numerically equivalent to an effective divisor, then the divisor class
\begin{equation} \label{eq:hat.H}
    \hat H \coloneqq 2H - (F_1 + T_1) - (F_2 + T_2) - (F_3 + T_3)
\end{equation}
is a Fano polarization. 
Furthermore, there exists a half-fiber \(\hat F_i\) appearing in the \(10\)-sequence associated to \(\hat H\) with tail \(\hat T_i\) such that
\[
    H - (F_j + T_j) - (F_k + T_k) \sim \hat F_i + \hat T_i
\]
for every permutation \((i,j,k)\) of \((1,2,3)\).
\end{proposition}
\begin{proof}
To begin with, we compute that \(\hat H^2 = 10\). Moreover, \(H.\hat H = 11\).

Next, we prove that \(\hat H\) is nef. Assume that there is a \((-2)\)-curve \(R\) with \(\hat H.R < 0\). 
Then, \(H.R > 0\) because \(\hat{H}.R_{i,j} \geq 0\) for every \((-2)\)-curve \(R_{i,j}\) appearing in the \(10\)-sequence associated to~\(H\).
We claim that \(H.R \leq 2\). First, we observe that \(h^2(\hat H - F_1) = h^0(F_1 - \hat H + K_X) = 0\) because \(F_1.(F_1 - \hat H + K_X) = -F_1.\hat H = F_1.T_1 - 4 < 0\), and that \((\hat H - F_1)^2 = 2 + 2F_1.T_1 \geq 2\), so \(h^0(\hat{H} - F_1) \geq 2\) by Riemann--Roch. The \((-2)\)-curve \(R\) is a fixed component of \(|\hat H - F_1|\), as \((\hat H - F_1).R \leq \hat H.R < 0\) by assumption. Thus, \(h^0(\hat H - F_1 - R) \geq 2\) as well. By \autoref{lem:Igors.trick}, it follows that \(H.(\hat H - F_1 - R)\geq 2\Phi(H) = 6\), whence \(H.R \leq 2\).

Therefore, we have either \(H.R = 1\) or \(H.R = 2\). 
In both cases, it holds that \((F_i + T_i).R \leq 1\) for each \(i\) by \Cref{lem:H.R=1.or.2}, since \(F_i+T_i\) is one of the divisors \(E_i\) appearing in the \(10\)-sequence associated to~\(H\). 
If \(H.R = 2\), then \[\hat H.R = 2H.R - \sum_{i=1}^3(F_i + T_i).R \geq 4-3 > 0,\] which contradicts our assumption. 
Thus, \(H.R = 1\), which implies \((F_i+T_i).R = 1\) for each \(i \in \{1,2,3\}\) by the same computation, since we assumed \(\hat{H}.R < 0\). 
By \Cref{lem:H.R=1.or.2}, \(R\) is orthogonal to all other divisors \(E_i\). A straightforward computation shows that also the divisor \(D = H - \sum_1^3 (F_i + T_i)\) satisfies \(D.(F_i + T_i) = 1\) for \(i \in \{1,2,3\}\) and \(D.E_i = 0\) for all other~\(E_i\). 
Since the classes of the divisors \(E_i\) generate \(\Num(X)\) over \(\IQ\), we have \(R \equiv D\). Hence, \(H - F_1 - F_2 - F_3\) is numerically equivalent to \(T_1 + T_2 + T_3 + R\), which contradicts the assumption of the proposition.

Next, we prove that \(\Phi(\hat H) = 3\). Let \(F\) be a half-fiber such that \(\hat H.F \leq 2\). 
Note that \(\hat H.F > 0\) by the Hodge index theorem. 
We have that \(h^2(\hat{H} - 2F)  = h^0(2F - \hat{H} + K_X) = 0\), since \(F.(2F-\hat{H})<0\), and \((\hat{H} - 2F)^2 = 10 - 4\hat{H}.F\geq 2\) by assumption, so \(h^0(\hat{H}-2F)\geq 2\) by Riemann--Roch. 
By \Cref{lem:Igors.trick}, we have \(H.(\hat{H}-2F) \geq 6\), from which we deduce \(H.F < 3\), a contradiction. Thus, by \Cref{def:Fano.polarization}, \(\hat H\) is a Fano polarization.

It remains to show the statement about the half-fibers occurring in the \(10\)-sequence associated to \(\hat{H}\). For this, let \(E = H - (F_j + T_j) - (F_k + T_k)\). Since \(E^2 = 0\) and \(E.F_i = 1\), the divisor \(E\) is linearly equivalent to an effective divisor by Riemann--Roch. According to \Cref{lem:reducibility}, there exists a nef, isotropic divisor \(\hat F_i\) and a (possibly empty) configuration \(\hat{T_i}\) of \((-2)\)-curves with \(E \sim \hat F_i + \hat T_i\). 
By \Cref{lem: bigandnefalwayscontainshalffiber}, the divisor \(\hat F_i\) is a half-fiber, because \(\Phi(H) = 3\) and \(H.\hat F_i \leq H.(\hat F_i + \hat T_i) = H.E = 4\).

We compute \(\hat H.(\hat F_i + \hat T_i) = \hat H.E = 3\). By \Cref{lem:isotropic.elements.in.10.sequence}, the divisor \(\hat F_i + \hat T_i\), so also the half-fiber~\(\hat F_i\), appears in the \(10\)-sequence associated to \(\hat H\). Observe that \(\hat H - (\hat F_i + \hat T_i) \sim H - (F_i + T_i)\). By~\Cref{prop:nefness.H-F} and the fact that \(T_i\) is the full tail of \(F_i\), we deduce that \(H - (F_i + T_i)\) is a nef divisor, hence \(\hat H - (\hat F_i + \hat T_i)\) is nef as well. By the same lemma, \(\hat T_i\) is the full tail of \(\hat F_i\).
\end{proof}

\begin{remark}
A closer look at the proof of \Cref{prop:nontrivial.tilt} shows that the divisor \(\hat H\) defined in~\eqref{eq:hat.H} is not nef if and only if \(H - F_1 - F_2 - F_3\) is numerically equivalent to an effective divisor. Moreover, in this case, there exists a \((-2)\)-curve \(R\) such that \(H \equiv F_1 + T_1 + F_2 + T_2 + F_3 + T_3 + R\), hence \(\hat{H} \equiv H + R\). 
In particular, \(H\) itself is the unique effective nef divisor class in the Weyl orbit of \(\hat{H}\).
\end{remark}

Recall that a \(3\)-sequence \((F_1,F_2,F_3)\) is called \emph{special} if \(F_2 + F_3 - F_1\) is numerically equivalent to an effective divisor.
For basic properties of special \(3\)-sequences, we refer to \cite[Section~3]{MMVnd3}.

\begin{theorem} \label{thm:special.3-sequence}
Every classical Enriques surface admitting a special \(3\)-sequence is a Reye congruence.
In particular, every nodal Enriques surface in characteristic \(p \neq 2\) is a Reye congruence.
\end{theorem}
\begin{proof}
Thanks to \Cref{thm:main_criterion}, it suffices to find a Fano polarization \(H\) such that \(H - 2(F_1+T_1)\) is numerically equivalent to an effective divisor, where \(T_1\) is the tail of a half-fiber \(F_1\) sitting in the associated \(10\)-sequence.

Complete the special \(3\)-sequence \((F_1,F_2,F_3)\) to a \(10\)-sequence with associated Fano polarization~\(H\) (see \cite[Lemma 1.6.1, Theorem 3.3]{Cossec:Picard_group}).
By definition, \(F_2 + F_3 - F_1\) is numerically equivalent to an effective divisor~\(S_1\). Observe that \(S_1 - T_1\) is effective as well. Indeed, if \(T_1 = R_{1,1} + \ldots + R_{1,r_1}\) for \(r_1 \geq 1\), then \(S_1.R_{1,1} = -1\), so \(S_1 - R_{1,1}\) is effective and, recursively, so is \(S_1 - T_1\).

Suppose first that \(H - F_1 - F_2 - F_3\) is numerically equivalent to an effective divisor \(D\). Arguing recursively as in the previous paragraph, we deduce that \(R = D - T_1 - T_2 - T_3\) is effective as well. Thus, we conclude by observing that \(H - 2(F_1 + T_1) \equiv (S_1 - T_1) + T_2 + T_3 + R\).

Suppose now that \(H - F_1 - F_2 - F_3\) is not numerically equivalent to an effective divisor. By \Cref{prop:nontrivial.tilt}, the divisor \(\hat H\) defined in \eqref{eq:hat.H} is a Fano polarization. 
Moreover, the associated \(10\)-sequence contains a half-fiber \(\hat F_1\) with tail~\(\hat T_1\) such that \(\hat H - 2(\hat F_1 + \hat T_1) \equiv (S_1 - T_1) + T_2 + T_3\), and we conclude.

The second part of the statement follows immediately from \cite[Theorem~3.11]{MMVnd3}.
\end{proof}

\begin{remark} \label{rk:geometric.intrepretation}
The proof of \Cref{thm:special.3-sequence} can be interpreted geometrically. Indeed, let \(\varphi_H \colon X \to \mathbb{P}^5\) be the morphism induced by a Fano polarization \(H\) on~\(X\) and let \(H_i\) (resp. \(H_i'\)) be the plane spanned by the image of \(F_i\) (resp. \(F_i'\)). The condition that \((F_1,F_2,F_3)\) is special, or more precisely that \(H^0(X,\cO_X(F_1+F_2-F_3)) \neq 0\), implies that the planes \(H_1,H_2,H_3\) meet in the point \(F_1 \cap F_2 \cap F_3\).

In this situation, the proof of \Cref{thm:special.3-sequence} shows that either \(\varphi_H\) maps \(X\) into a smooth quadric in \(\IP^5\) (if \(H-F_1-F_2-F_3\) is numerically equivalent to an effective divisor), or there exists an explicit standard Cremona transformation of \(\IP^5\) mapping \(X\) into a smooth quadric. More precisely, in this second case there exists a commutative diagram
\[
\begin{tikzcd}
    X \arrow{r}{\varphi_H} \arrow{dr}[swap]{\varphi_{\hat{H}}} &\IP^5 \arrow[dashed]{d}{\varphi}\\
    &\IP^5
\end{tikzcd}
\]
where \(\varphi \colon \IP^5 \dashrightarrow \IP^5\) is the Cremona transformation induced by the linear system \(|2H-F_1-F_2-F_3|\) on \(\IP^5\) (cf. \cite[Remark~6.23]{conte.verra}), and the morphism \(\varphi_{\hat{H}}\) induced by \(|\hat{H}|\) maps \(X\) into a smooth quadric.
\end{remark}

\section{Fano--Reye polarizations on extra-special Enriques surfaces} \label{sec:extraspecial}

In this section, we conclude the proof of \Cref{thm:main.theorem.introduction}. In \cite[Theorem~3.11]{MMVnd3}, we showed that every nodal, non-extra-special Enriques surface admits a special \(3\)-sequence. By \Cref{thm:special.3-sequence}, it only remains to settle \Cref{thm:main.theorem.introduction} for the three families of classical extra-special Enriques surfaces in characteristic~\(2\). For more details on extra-special Enriques surfaces, we refer to \cite{MMV:extraspecial}.

\begin{theorem} \label{thm:extra-special}
Every classical extra-special Enriques surface is a Reye congruence.
\end{theorem}
\begin{proof}
Thanks to \Cref{thm:main_criterion}, it suffices to find a \(10\)-sequence with Fano polarization \(H\) such that \(H - 2(F_i+T_i)\) is numerically equivalent to an effective divisor for some \(i\), where \(T_i\) is the tail of a half-fiber \(F_i\) sitting in the \(10\)-sequence.

First, let \(X\) be a classical, extra-special surface of type \(\tilde E_8\). By definition, the dual graph of \((-2)\)-curves on~\(X\) is the following:
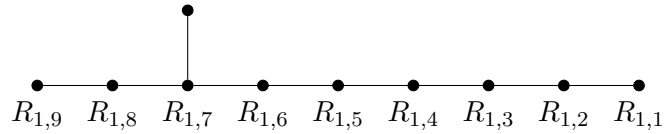
\[
\begin{tikzpicture}
    \node (R4) at (0,0) [nodal, label = below:\(R_{1,7}\)] {};
    \node (R5) at (0,1) [nodal] {};
    \node (R6) at (1,0) [nodal, label = below:\(R_{1,6}\)] {};
    \node (R7) at (2,0) [nodal, label = below:\(R_{1,5}\)] {};
    \node (R8) at (3,0) [nodal, label = below:\(R_{1,4}\)] {};
    \node (R9) at (4,0) [nodal, label = below:\(R_{1,3}\)] {};
    \node (RX) at (5,0) [nodal, label = below:\(R_{1,2}\)] {};
    \node (R3) at (-1,0) [nodal, label = below:\(R_{1,8}\)] {};
    \node (R2) at (-2,0) [nodal, label = below:\(R_{1,9}\)] {};
    \node (R11) at (6,0) [nodal, label = below:\(R_{1,1}\)] {};
    \draw (R2)--(R3)--(R4) (R5)--(R4)--(RX) (RX)--(R11);
\end{tikzpicture}
\]
Consider the \(1\)-degenerate \(10\)-sequence containing  the curves \(R_{1,j}\) as indicated in the picture above, and the half-fiber \(F_1\) of type \(\II^*\) (cf. \cite[Theorem 8.10.41]{DolgachevKondoBook}). 
Let \(H\) be the corresponding Fano polarization.
One computes that the divisor \(H - 2(F_1 + R_{1,1} + \ldots + R_{1,9})\) is numerically equivalent to the following effective divisor:
\begin{equation} \label{eq:case.E8.extra.special}
\begin{tabular}{c}
\begin{tikzpicture}
    \node (R4) at (0,0) [nodal, label = below: {\(7\)}] {};
    \node (R5) at (0,1) [nodal, label = right: {\(4\)}] {};
    \node (R6) at (1,0) [nodal, label = below: {\(6\)}] {};
    \node (R7) at (2,0) [nodal, label = below: {\(5\)}] {};
    \node (R8) at (3,0) [nodal, label = below: {\(4\)}] {};
    \node (R9) at (4,0) [nodal, label = below: {\(3\)}] {};
    \node (RX) at (5,0) [nodal, label = below: {\(2\)}] {};
    \node (R3) at (-1,0) [nodal, label = below: {\(4\)}] {};
    \node (R2) at (-2,0) [nodal, label = below: {\(1\)}] {};
    \node (R11) at (6,0) [nodal, label = below: {\(1\)}] {};
    \draw (R2)--(R3)--(R4) (R5)--(R4)--(RX) (RX)--(R11);
\end{tikzpicture}
\end{tabular}
\end{equation}

Let now \(X\) be a classical, extra-special surface of type \(\tilde D_8\). By definition, the dual graph of \((-2)\)-curves on~\(X\) is the following:
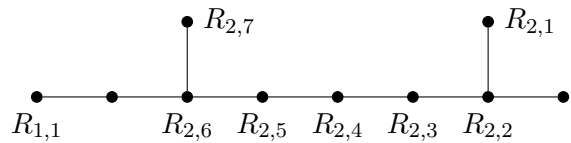
\[
\begin{tikzpicture}     \node (R4) at (0,0) [nodal, label=below:\(R_{2,6}\)] {};
    \node (R5) at (0,1) [nodal, label=right:\(R_{2,7}\)] {};
    \node (R6) at (1,0) [nodal, label=below:\(R_{2,5}\)] {};
    \node (R7) at (2,0) [nodal, label=below:\(R_{2,4}\)] {};
    \node (R8) at (3,0) [nodal, label=below:\(R_{2,3}\)] {};
    \node (R9) at (4,0) [nodal, label=below:\(R_{2,2}\)] {};
    \node (RX) at (5,0) [nodal] {};
    \node (R3) at (-1,0) [nodal] {};
    \node (R2) at (-2,0) [nodal, label=below:\(R_{1,1}\)] {};
    \node (R1) at (4,1) [nodal, label=right:\(R_{2,1}\)] {};
\draw (R2)--(R3)--(R4) (R5)--(R4)--(RX) (R1)--(R9);
\end{tikzpicture}
\]
Consider the \(2\)-degenerate \(10\)-sequence containing the curves \(R_{i,j}\) as indicated in the picture above, the half-fiber \(F_1\) of type \(\I_4^*\), and a half-fiber \(F_2\) such that \(|2F_2|\) has a simple fiber \(G_2\) of type \(\II^*\) as indicated below (cf. \cite[Theorem 8.10.70]{DolgachevKondoBook}).
\[
\begin{tikzpicture}[scale=0.6]
    \node (R4) at (0,0) [nodal] {};
    \node (R5) at (0,1) [nodal] {};
    \node (R6) at (1,0) [nodal] {};
    \node (R7) at (2,0) [nodal] {};
    \node (R8) at (3,0) [nodal] {};
    \node (R9) at (4,0) [nodal] {};
    \node (RX) at (5,0) [nodal] {};
    \node (R3) at (-1,0) [nodal] {};
    \node (R2) at (-2,0) [nodal,fill=white,label=left:\(F_1\)] {};
    \node (R1) at (4,1) [nodal] {};
    \draw (R2)--(R3);
    \draw [very thick, densely dashed] (R3)--(R4) (R5)--(R4)--(RX) (R1)--(R9);
\end{tikzpicture}
\qquad
\begin{tikzpicture}[scale=0.6]
    \node (R4) at (0,0) [nodal] {};
    \node (R5) at (0,1) [nodal] {};
    \node (R6) at (1,0) [nodal] {};
    \node (R7) at (2,0) [nodal] {};
    \node (R8) at (3,0) [nodal] {};
    \node (R9) at (4,0) [nodal] {};
    \node (RX) at (5,0) [nodal] {};
    \node (R3) at (-1,0) [nodal] {};
    \node (R2) at (-2,0) [nodal,label=left:\(G_2\)] {};
    \node (R1) at (4,1) [nodal,fill=white] {};
    \draw [very thick] (R2)--(R3)--(R4) (R5)--(R4)--(RX);
    \draw (R1)--(R9);
\end{tikzpicture}
\]
Let \(H\) be the corresponding Fano polarization. 
One computes that \(H-2(F_1+R_{1,1})\) is numerically equivalent to the following effective divisor:
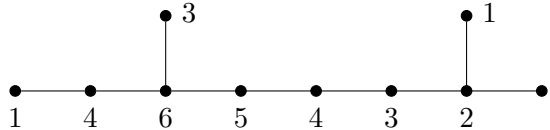
\[
\begin{tikzpicture}     
    \node (R4) at (0,0) [nodal, label=below: \(6\)] {};
    \node (R5) at (0,1) [nodal, label=right: \(3\)] {};
    \node (R6) at (1,0) [nodal, label=below: \(5\)] {};
    \node (R7) at (2,0) [nodal, label=below: \(4\)] {};
    \node (R8) at (3,0) [nodal, label=below: \(3\)] {};
    \node (R9) at (4,0) [nodal, label=below: \(2\)] {};
    \node (RX) at (5,0) [nodal] {};
    \node (R3) at (-1,0) [nodal, label=below: \(4\)] {};
    \node (R2) at (-2,0) [nodal, label=below: \(1\)] {};
    \node (R1) at (4,1) [nodal, label=right: \(1\)] {};
    \draw (R2)--(R3)--(R4) (R5)--(R4)--(RX) (R1)--(R9);
\end{tikzpicture}
\]

Finally, let \(X\) be a classical, extra-special surface of type \(\tilde E_7\). By definition, the dual graph of \((-2)\)-curves on~\(X\) is the following:
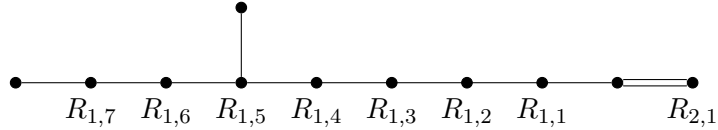
\[
\begin{tikzpicture}
    \node (R4) at (0,0) [nodal,label=below: \(R_{1,5}\)] {};
    \node (R5) at (0,1) [nodal] {};
    \node (R6) at (1,0) [nodal,label=below: \(R_{1,4}\)] {};
    \node (R7) at (2,0) [nodal,label=below: \(R_{1,3}\)] {};
    \node (R8) at (3,0) [nodal,label=below: \(R_{1,2}\)] {};
    \node (R9) at (4,0) [nodal,label=below: \(R_{1,1}\)] {};
    \node (RX) at (5,0) [nodal] {};
    \node (R3) at (-1,0) [nodal,label=below: \(R_{1,6}\)] {};
    \node (R2) at (-2,0) [nodal,label=below: \(R_{1,7}\)] {};
    \node (R1) at (-3,0) [nodal] {};
    \node (R11) at (6,0) [nodal,label=below: \(R_{2,1}\)] {};
    \draw (R1)--(R2)--(R3)--(R4) (R5)--(R4)--(RX);
    \draw[double] (RX)--(R11);
\end{tikzpicture}
\]
Consider the \(10\)-sequence containing the curves \(R_{i,j}\) as indicated in the picture above, the half-fiber \(F_1\) of type \(\III^*\), and a half-fiber \(F_2\) such that \(|2F_2|\) has a simple fiber \(G_2\) of type \(\II^*\) as indicated below (cf. \cite[Proposition 8.10.48]{DolgachevKondoBook}). 
\[
\begin{tikzpicture}[scale=0.6]
    \node (R4) at (0,0) [nodal] {};
    \node (R5) at (0,1) [nodal] {};
    \node (R6) at (1,0) [nodal] {};
    \node (R7) at (2,0) [nodal] {};
    \node (R8) at (3,0) [nodal] {};
    \node (R9) at (4,0) [nodal,fill=white] {};
    \node (RX) at (5,0) [nodal,fill=white] {};
    \node (R3) at (-1,0) [nodal] {};
    \node (R2) at (-2,0) [nodal] {};
    \node (R1) at (-3,0) [nodal,label=left:\(F_1\)] {};
    \node (R11) at (6,0) [nodal,fill=white] {};
    \draw [very thick, densely dashed] (R1)--(R2)--(R3)--(R4) (R5)--(R4)--(R6)--(R7)--(R8);
    \draw (R8)--(R9)--(RX);
    \draw[double] (RX)--(R11);
\end{tikzpicture}
\]\[
\begin{tikzpicture}[scale=0.6]
    \node (R4) at (0,0) [nodal] {};
    \node (R5) at (0,1) [nodal] {};
    \node (R6) at (1,0) [nodal] {};
    \node (R7) at (2,0) [nodal] {};
    \node (R8) at (3,0) [nodal] {};
    \node (R9) at (4,0) [nodal] {};
    \node (RX) at (5,0) [nodal] {};
    \node (R3) at (-1,0) [nodal] {};
    \node (R2) at (-2,0) [nodal] {};
    \node (R1) at (-3,0) [nodal,fill=white,label=left:\(G_2\)] {};
    \node (R11) at (6,0) [nodal,fill=white] {};
    \draw (R1)--(R2);
    \draw [very thick] (R2)--(R3)--(R4) (R5)--(R4)--(RX);
    \draw[double] (RX)--(R11);
\end{tikzpicture}
\]
Let \(H\) be the corresponding Fano polarization.
One computes that \(H-2(F_2 + R_{2,1})\) is numerically equivalent to the following effective divisor:
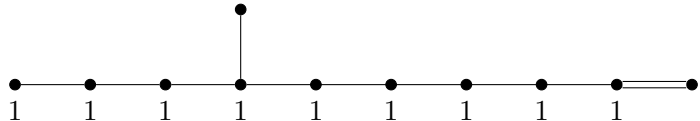
\[
\begin{tikzpicture}
    \node (R4) at (0,0) [nodal,label=below: \(1\)] {};
    \node (R5) at (0,1) [nodal] {};
    \node (R6) at (1,0) [nodal,label=below: \(1\)] {};
    \node (R7) at (2,0) [nodal,label=below: \(1\)] {};
    \node (R8) at (3,0) [nodal,label=below: \(1\)] {};
    \node (R9) at (4,0) [nodal,label=below: \(1\)] {};
    \node (RX) at (5,0) [nodal,label=below: \(1\)] {};
    \node (R3) at (-1,0) [nodal,label=below: \(1\)] {};
    \node (R2) at (-2,0) [nodal,label=below: \(1\)] {};
    \node (R1) at (-3,0) [nodal,label=below: \(1\)] {};
    \node (R11) at (6,0) [nodal] {};
    \draw (R1)--(R2)--(R3)--(R4) (R5)--(R4)--(RX);
    \draw[double] (RX)--(R11);
\end{tikzpicture} \qedhere
\]
\end{proof}

\section{A nodal Enriques surface of non-degeneracy 10 without~ample~Fano--Reye~polarizations} \label{sec:counterexample}

As explained in \Cref{rmk:R(W)} in the introduction, Enriques surfaces arising via the Reye construction from a \emph{regular} web of quadrics necessarily admit an ample Fano--Reye polarization, and in particular their nondegeneracy is \(10\). Conversely, it was claimed in \cite[p.~40]{Dolgachev.Reider} that every Enriques surface of nondegeneracy \(10\) admits an ample Fano--Reye polarization.

In this section, we provide a counterexample to this claim. In particular, we show that if \(X\) is an Enriques surface of Kond\={o}'s type~\(\mathrm{VII}\), then \(X\) admits no ample Fano--Reye polarization, despite the fact that it has non-degeneracy \(10\) \cite[§7.3.7]{Moschetti.Rota.Schaffler}. Recall from \cite[Figure 8.15]{DolgachevKondoBook} that the dual graph of \((-2)\)-curves on~\(X\), labelled according to \cite[Figure 7.7]{Kondo:Enriques.finite.aut}, looks as follows:

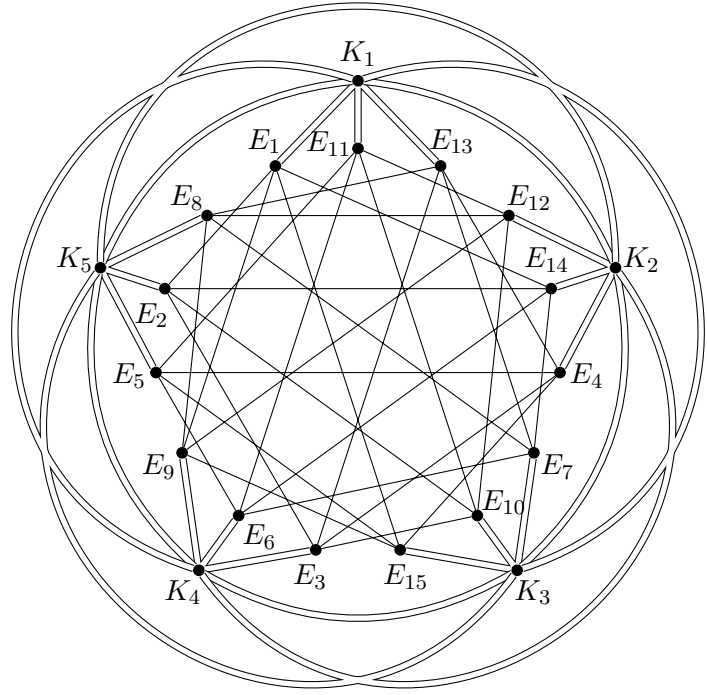
\[
    \begin{tikzpicture}
\begin{scope}[scale=0.9]
    \useasboundingbox (-5,-5.2) rectangle (5,5.2);
    
    \foreach \x in {1,2,3,4,5,6,7,8,9,10,11,12,13,14,15}{
            \node (A\x) at (138-24*\x:3) [nodal] {};
        };
    \foreach \x in {1,2,3,4,5}{
            \node (B\x) at (162-72*\x:4) [nodal] {};
        };
    
    \draw (A1)--(A5) (A1)--(A9) (A1)--(A12) (A1)--(A14)
          (A2)--(A4) (A2)--(A8) (A2)--(A11) (A2)--(A13)
          (A3)--(A6) (A3)--(A7) (A3)--(A10) (A3)--(A15)
          (A4)--(A8) (A4)--(A12) (A4)--(A15)
          (A5)--(A7) (A5)--(A11) (A5)--(A14)
          (A6)--(A9) (A6)--(A10) (A6)--(A13)
          (A7)--(A11) (A7)--(A15)
          (A8)--(A10) (A8)--(A14)
          (A9)--(A12) (A9)--(A13)
          (A10)--(A14) 
          (A11)--(A13)
          (A12)--(A15);

    \draw [double]  (B1)--(A1) (B1)--(A2) (B1)--(A3)
                    (B2)--(A4) (B2)--(A5) (B2)--(A6)
                    (B3)--(A7) (B3)--(A8) (B3)--(A9)
                    (B4)--(A10) (B4)--(A11) (B4)--(A12)
                    (B5)--(A13) (B5)--(A14) (B5)--(A15);
    \draw [double,in=150,out=30,relative] (B1) to (B2) (B2) to (B3) (B3) to (B4) (B4) to (B5) (B5) to (B1);
    \draw [double,in=90,out=90,relative,looseness=1.7] (B1) to (B3) (B3) to (B5) (B5) to (B2) (B2) to (B4) (B4) to (B1);

    \node at (90:4.4) {\(K_1\)};
    \node at (18:4.4) {\(K_2\)};
    \node at (-54:4.4) {\(K_3\)};
    \node at (-126:4.4) {\(K_4\)};
    \node at (162:4.4) {\(K_5\)};

    \node at (114:3.4) {\(E_1\)};
    \node at (98:3.1) {\(E_{11}\)};
    \node at (66:3.4) {\(E_{13}\)};
    \node at (42:3.4) {\(E_{12}\)};
    \node at (26:3.1) {\(E_{14}\)};
    \node at (-6:3.4) {\(E_4\)};
    \node at (-30:3.4) {\(E_7\)};
    \node at (-46:3.1) {\(E_{10}\)};
    \node at (-78:3.4) {\(E_{15}\)};
    \node at (-102:3.4) {\(E_3\)};
    \node at (-118:3.1) {\(E_6\)};
    \node at (-150:3.4) {\(E_9\)};
    \node at (-174:3.4) {\(E_5\)};
    \node at (170:3.1) {\(E_2\)};
    \node at (138:3.4) {\(E_8\)};

\end{scope}    
\end{tikzpicture}
\]

\begin{lemma} \label{lem:Fano.polarizations.type.VII}
    Every ample Fano polarization on~\(X\) is numerically equivalent to one of the divisors
    \[H_n \coloneqq \frac{1}{3}\sum_{i=1}^{15}{E_i} + \frac{1}{6}\sum_{j=1}^5{K_j} + \frac{1}{2}K_n \]
    for \(n=1,\ldots,5\). Moreover, the group \(\Aut(X)\cong \mathfrak{S}_5\) acts transitively on the \(H_i\).
\end{lemma}
\begin{proof}
By \cite[§7.3.7]{Moschetti.Rota.Schaffler}, there are precisely \(5\) ample Fano polarizations on~\(X\): we claim that these are indeed \(H_1,\ldots,H_5\). Moreover, we know by \cite[Table~2]{Moschetti.Rota.Schaffler} that the \((-2)\)-curves on~\(X\), together with the divisor \(\frac{1}{2}(E_1+E_2+E_3+E_4+E_{15})\), generate the numerical lattice \(\Num(X)\).
A computation shows that the divisors \(H_n\) meet these divisors with integer multiplicity, so \(H_n \in \Num(X)^\vee =\Num(X)\) by the fact that \(\Num(X)\) is unimodular. Furthermore, \(H_n^2 = 10\) and \(H_n.E_i, H_n.K_j >0\) for all indices \(n\), \(i\) and~\(j\), so the divisors \(H_n\) are indeed ample.

It remains to show that \(\Phi(H_n)=3\). Since every elliptic fibration on~\(X\) admits a reducible fiber with at least \(5\) components (cf. \cite[Proposition~8.9.33]{DolgachevKondoBook}) and the \(H_n\) are ample, it follows that \(H_n . G \ge 5\) for every elliptic pencil \(|G|\). Thus, \(H_n . F > 2\) for every half-fiber \(F\) on~\(X\). Since \(\Phi(H_n) ^2 \leq H_n^2 = 10\) by \cite[Proposition~2.4.11]{CossecDolgachevLiedtke}, this implies \(\Phi(H_n) = 3\).

For the last statement, recall that the automorphism group of \(X\) acts transitively on the \(5\) outer vertices \(K_n\) (see \cite[p.~233]{Kondo:Enriques.finite.aut}), and, therefore, it acts transitively on the \(5\) divisors \(H_n\).
\end{proof}

\begin{proposition}
There are no ample Fano--Reye polarizations on an Enriques surface \(X\) of type~\(\mathrm{VII}\). In particular, \(X\) is not a classical Reye congruence.
\end{proposition}
\begin{proof}
By \Cref{lem:Fano.polarizations.type.VII}, we know that \(\Aut(X)\) permutes the ample Fano polarizations on~\(X\), so it suffices to show that \(H_1\) is not a Fano--Reye polarization. Equivalently, by \cite[Lemma~7.8.1]{DolgachevKondoBook}, it suffices to exhibit a half-fiber \(F\) on~\(X\) such that \(H_1.F=3\) and \(H_1 - 2F+K_X\) is not linearly equivalent to an effective divisor.

Consider the elliptic fiber \(G \coloneqq E_1+E_2+E_9+E_{10}+E_{12}\). It intersects the divisors generating \(\Num(X)\) with even multiplicity, so \(G\) is a simple fiber on~\(X\) and there exists a half-fiber \(F\) such that \(G\sim 2F\). One computes that \(H_1.F=3\).

Now, assume by contradiction that \(D \coloneqq H_1-G\) is numerically equivalent to an effective divisor. We have \(D.E_i < 0\) for \(i\in \{3,8,14,15\}\), so \(D' \coloneqq D- E_3-E_8-E_{14}-E_{15}\) is numerically equivalent to an effective divisor as well. However, \(D'\) intersects the elliptic fiber \(G' \coloneqq \sum_{i=1}^9 E_i\) negatively, contradicting the nefness of \(G'\).
\end{proof}

\bibliographystyle{amsplain}
\bibliography{Enriques}

\providecommand{\bysame}{\leavevmode\hbox to3em{\hrulefill}\thinspace}
\providecommand{\MR}{\relax\ifhmode\unskip\space\fi MR }
\providecommand{\MRhref}[2]{%
  \href{http://www.ams.org/mathscinet-getitem?mr=#1}{#2}
}
\providecommand{\href}[2]{#2}
\begin{thebibliography}{10}

\bibitem{Brivio}
Sonia Brivio, \emph{Smooth {E}nriques surfaces in {${\bf P}^4$} and exceptional
  bundles}, Math. Z. \textbf{213} (1993), no.~4, 509--521. \MR{1231875}

\bibitem{conte.verra}
Alberto Conte and Alessandro Verra, \emph{Reye constructions for nodal
  {E}nriques surfaces}, Trans. Amer. Math. Soc. \textbf{336} (1993), no.~1,
  79--100. \MR{1079052}

\bibitem{Cossec:Reye}
Fran\c{c}ois Cossec, \emph{Reye congruences}, Trans. Amer. Math. Soc.
  \textbf{280} (1983), no.~2, 737--751. \MR{716848}

\bibitem{Cossec:Picard_group}
\bysame, \emph{On the {P}icard group of {E}nriques surfaces}, Math. Ann.
  \textbf{271} (1985), no.~4, 577--600. \MR{790116}

\bibitem{Cossec.Dolgachev}
Fran\c{c}ois Cossec and Igor Dolgachev, \emph{\rm{Enriques surfaces. {I}}},
  vol.~76, Birkh\"auser Boston, Inc., Boston, MA, 1989. \MR{986969}

\bibitem{CossecDolgachevLiedtke}
Fran\c{c}ois Cossec, Igor Dolgachev, and Christian Liedtke, \emph{Enriques
  surfaces {I}}, with an appendix by Shigeyuki Kond\={o}, draft available on
  \url{www.math.lsa.umich.edu/~idolga/EnriquesOne.pdf}, 2023.

\bibitem{darboux}
Gaston Darboux, \emph{Sur les syst\`emes lin\'eaires de coniques et de surfaces
  du second ordre}, Bulletin des Sciences Math\'ematiques et Astronomiques
  \textbf{1} (1870), 348--358.

\bibitem{Dolgachev.Whatisleft}
Igor Dolgachev, \emph{Enriques surfaces: what is left?}, Problems in the theory
  of surfaces and their classification ({C}ortona, 1988), Sympos. Math., XXXII,
  Academic Press, London, 1991, pp.~81--97. \MR{1273373}

\bibitem{DolgachevKondoBook}
Igor Dolgachev and Shigeyuki Kond\=o, \emph{Enriques surfaces {II}}, draft
  available on \url{www.math.lsa.umich.edu/~idolga/EnriquesTwo.pdf}, 2023.

\bibitem{Dolgachev.Reider}
Igor Dolgachev and Igor Reider, \emph{On rank {$2$} vector bundles with
  {$c^2_1=10$} and {$c_2=3$} on {E}nriques surfaces}, Algebraic geometry
  ({C}hicago, {IL}, 1989), Lecture Notes in Math., vol. 1479, Springer, Berlin,
  1991, pp.~39--49. \MR{1181205}

\bibitem{Fano}
Gino Fano, \emph{Superficie algebriche di genere zero e bigenere uno, e loro
  casi particolari}, Rendiconti del Circolo Matematico di Palermo (1884--1940)
  \textbf{29} (1910), no.~1, 98--118.

\bibitem{Kim:exc.bundles.nodal}
Hoil Kim, \emph{Exceptional bundles on nodal {E}nriques surfaces}, Manuscripta
  Math. \textbf{82} (1994), no.~1, 1--13. \MR{1254135}

\bibitem{Kondo:Enriques.finite.aut}
Shigeyuki Kond\={o}, \emph{Enriques surfaces with finite automorphism groups},
  Japan. J. Math. (N.S.) \textbf{12} (1986), no.~2, 191--282. \MR{914299}

\bibitem{MMVnd3}
Gebhard Martin, Giacomo Mezzedimi, and Davide~Cesare Veniani, \emph{{Enriques
  surfaces of non-degeneracy 3}}, preprint,
  \href{https://arxiv.org/abs/2203.08000}{arXiv:2203.08000}, 2022.

\bibitem{MMV:extraspecial}
\bysame, \emph{On extra-special {E}nriques surfaces}, Math. Ann. (2022).

\bibitem{Moschetti.Rota.Schaffler}
Riccardo Moschetti, Franco Rota, and Luca Schaffler, \emph{A computational view
  on the non-degeneracy invariant for {E}nriques surfaces}, Exp. Math. (2022).

\bibitem{Naie:special}
Daniel Naie, \emph{Special rank two vector bundles over {E}nriques surfaces},
  Math. Ann. \textbf{300} (1994), no.~2, 297--316. \MR{1299064}

\bibitem{Reye}
Theodor Reye, \emph{Die {G}eometrie der {L}age}, 3rd ed., vol.~3,
  Baumgärtner's Buchhandlung, Leipzig, 1892.

\bibitem{Zube:exc.vec.bundles}
Severinas Zub\.{e}, \emph{Exceptional vector bundles on {E}nriques surfaces},
  Mat. Zametki \textbf{61} (1997), no.~6, 825--834. \MR{1629793}

\end{thebibliography}

\end{document}